\pdfoutput=1

\documentclass[11pt]{amsart}

\usepackage{amsmath, amstext, amsbsy, amsopn, amsfonts}
\usepackage{subfig}
\usepackage{graphicx}
\usepackage{doi}
\usepackage[margin=1.1in]{geometry}
\usepackage{hhline}

\usepackage{tikz}
\usepackage{algpseudocode}
\usepackage{enumerate}
\usepackage{algorithm} 
\usepackage{cite}

\newcommand{\R}{\mathbb{R}}
\DeclareMathOperator \conv {conv} 
\DeclareMathOperator \TSP {TSP} 
\DeclareMathOperator \PYR {PYR}

\begin{document}

\theoremstyle{plain}
\newtheorem{theorem}{Theorem}
\newtheorem{lemma}{Lemma}
\newtheorem{conjecture}{Conjecture}

\title{On the skeleton of the pyramidal tours polytope}

\author{Vladimir~A.~Bondarenko, Andrei~V.~Nikolaev}
\thanks {The research was supported by the initiative R\&D VIP-004 YSU AAAA-A16-116070610022-6.}

\address{%
Department of Discrete Analysis, P.G. Demidov Yaroslavl State University, Sovetskaya, 14, Yaroslavl, 150000, Russia
}
\email {bond@bond.edu.yar.ru, andrei.v.nikolaev@gmail.com}

\begin{abstract}
We consider the skeleton of the pyramidal tours polytope. 
Hamiltonian tour is called pyramidal if the salesperson starts in city $1$, then visits some cities in increasing order, reaches city $n$ and returns to city $1$, visiting the remaining cities in decreasing order. 
The polytope $\PYR(n)$ is defined as the convex hull of characteristic vectors of all pyramidal tours in the complete graph $K_{n}$.
The skeleton of the polytope $\PYR(n)$ is the graph whose vertex set is the vertex set of $\PYR(n)$ and edge set is the set of geometric edges or one-dimensional faces of $\PYR(n)$.
We describe the necessary and sufficient condition for the adjacency of vertices of the polytope $\PYR(n)$. 
On this basis we developed an algorithm to check the vertex adjacency with a linear complexity.
We establish that the diameter of $\PYR(n)$ skeleton equals 2, and the asymptotically exact estimate of $\PYR(n)$ skeleton's clique number is $\Theta(n^{2})$. 
It is known that this value characterizes the time complexity in a broad class of algorithms based on linear comparisons.
\end {abstract}


\maketitle


\section*{Introduction}
\indent

We consider the classic instance of the symmetric traveling salesperson problem: for a given complete weighted undirected graph $K_{n}$, find a Hamiltonian cycle with a minimum weight.
We denote by $E$ the set of edges of the complete graph $K_{n}$, and by $HC_{n}$ the set of all Hamiltonian cycles in $K_{n}$.
With each Hamiltonian cycle $x \in HC_{n}$ we associate a characteristic vector $x^{v} \in \R^{E}$ by the following rule:

\begin{gather*}
x^{v}_{e} =
\begin{cases}
1, & \mbox {if an edge } e \mbox { is contained in the cycle } x,\\
0, & \mbox {otherwise.}
\end{cases}
\end{gather*}

The polytope
$$\TSP(n) = \conv \{x^{v}\ |\ x \in HC_{n}\}$$
is called the symmetric traveling salesperson polytope.

A partial description of the traveling salesperson polytope is used in algorithms based on integer linear programming methods through which the main record exact results were obtained for large traveling salesperson problems.
Including the classic result of Danzig, Falkerson, and Johnson for 49 US cities \cite{Dantzig} and the best route for the largest pla85900 instance on $85,900$ cities from the TSPLIB library, emerging in the VLSI design and formulated by Johnson during his work at AT\&T labs \cite{Applegate}.

In this paper we consider a skeleton of a polytope, also known as $1$-skeleton. 
The skeleton of a polytope $P$ is the graph whose vertex set is the vertex set of $P$ (characteristic vectors $x^{v}$ for the traveling salesperson problem) and edge set is the set of geometric edges or one-dimensional faces of $P$. 
A significant number of works are devoted to the study of $\TSP(n)$ skeleton. 
This is due both to the applied importance of the problem and to the complexity of the associated polytope. 
In particular, the classic result of Papadimitriou says that even construction of $\TSP(n)$ skeleton is a hard problem.

\begin{theorem} [Papadimitriou \cite {Papadimitriou}] \label {TSP_adjacency_NPC}
	The question whether two vertices of $\TSP(n)$ are nonadjacent is NP-complete.
\end{theorem}

Despite this fact, some properties of the skeleton of the traveling salesperson polytope have been established. 
In particular, the following two characteristics were studied: the diameter of a graph $G$, denoted by $d(G)$, is the maximum edge distance between any pair of vertices, and the clique number of a graph $G$, denoted by $\omega(G)$, is the number of vertices in a maximum clique of $G$. 

The study of $1$-skeleton's diameter is motivated by its relationship to edge-following algorithms of linear programming such as the simplex method (the diameter serves as the lower bound on the number of non-degenerate steps) and the well-known Hirsch conjecture.
Gr\"{o}tchel and Padberg, based on a complete description of the small dimension traveling salesperson polytopes ($n \leq 9$) and the fact that, for an asymmetric problem, the diameter of $1$-skeleton equals $2$ \cite{Padberg}, made the following assumption.

\begin{conjecture} [Gr\"{o}tschel, Padberg \cite {Grotschel}] \label {conjecture_TSP_diameter}
$$\forall n\geq 5: \ d(\TSP(n)) = 2.$$
\end{conjecture}

The conjecture remains open.
In a series of papers, the consistently improving upper bounds were constructed \cite {Rispoli, Sierksma, Sierksma-Teunter}, and the different faces of the traveling salesperson polytope were studied \cite{Sierksma-Teunter, Sierksma-Tijssen}.
The best upper estimate at the moment is $4$ \cite{Rispoli}.
The diameter of the generalization of the traveling salesperson polytope -- the polytope of $k$-cycles, that is the convex hull of characteristic vectors of all cycles on $k$ vertices in the complete graph $K_{n}$, was considered in the paper \cite{Girlich}.

The clique number of the skeleton serves as a lower bound for computational complexity in a class of direct-type algorithms based on linear comparisons.
In addition, it was found that this characteristic is polynomial for known polynomially solvable problems and is superpolynomial for intractable problems (see, for example, \cite {Bondarenko-Maksimenko, Bondarenko-NS, Bondarenko-Nikolaev-2016}).

Thus, for the symmetric traveling salesperson polytope, a lower bound on the clique number of a skeleton is superpolynomial in dimension.

\begin{theorem} [Bondarenko \cite {Bondarenko}] \label {theorem_TSP_clique}
	The clique number of $\TSP(n)$ skeleton is superpolynomial in dimension:
	$$\omega(\TSP(n)) \geq 2^{\left(\sqrt {\left\lfloor \frac {n}{2} \right\rfloor} - 9\right) \slash 2}.$$
\end{theorem}

A significant number of works are devoted to the study of polynomially solvable cases of the traveling salesperson problem (see, for example, the surveys \cite {Burkard, Gilmore}).
One of the most important classes of this kind is pyramidal tours.

Hamiltonian tour 
$$\phi = \left\langle 1,i_{1},i_{2},\ldots i_{r}, n, j_{1},j_{2},\ldots j_{n-r-2}\right\rangle$$
is called pyramidal if
$$i_{1} < i_{2} < \ldots < i_{r} \ \mbox {и} \ j_{1} > j_{2} > \ldots > j_{n-r-2}.$$
In other words, the salesperson starts in city $1$, then visits some cities in increasing order, reaches city $n$ and returns to city $1$, visiting the remaining cities in decreasing order.
Pyramidal tours have two nice properties. First, a minimum cost pyramidal tour can be determined in $O(n^2)$ time by dynamic programming, while the total number of pyramidal tours is exponential in $n$ \cite {Klyaus, Gilmore}. Second, there exist certain combinatorial structures of distance matrices that guarantee the existence of the shortest tour that is pyramidal.

For the first time, constraints on the distance matrices associated with polynomially solvable pyramidal cases of the traveling salesperson problem were studied by Aizenshtat and Kravchuk \cite {Aizenshtat}, where a particular case of the Monge matrices was considered.
Gilmour, Lawler, and Shmoys generalized this result to the Monge matrices of an arbitrary form \cite {Gilmore}.
Later, various classes of distance matrices with similar properties were described, including Van der Veen matrices, Demidenko matrices, Kalmanson matrices, Supnick matrices, and many others \cite {Burkard, Gilmore}.
A complete classification of the constraints on the distance matrices over four points, where the restriction is imposed on the distances between any four cities that generate a polynomially solvable instances of the traveling salesperson problem, is given in \cite {Deineko}.

However, the polyhedral characteristics of pyramidal tours and their relation to the general traveling salesperson problem have never before been the object of a direct research.

The results of this paper were presented at the 17th Baikal international school-seminar ``Methods of Optimization and Their Applications'', Maksimikha, Buryatia, July 31 -- August 6, 2017 \cite {Bondarenko-Nikolaev-BITSS} and the European conference on combinatorics, graph theory and applications (Eurocomb 2017), Vienna, Austria, August 28 -- September 1, 2017 \cite{Bondarenko-Nikolaev-Eurocomb}.

\section {The pyramidal tours polytope}
\indent

We consider a complete weighted undirected graph $K_{n}$ with the edge set $E$. Let $PT_{n}$ be the set of all pyramidal tours in $K_{n}$. With each pyramidal tour $x \in PT_{n}$ we associate a characteristic vector $x^{v} \in \R^{E}$ by the following rule:
\begin{gather*}
x^{v}_{e} =
\begin{cases}
1,& \mbox {if an edge } e \mbox { is contained in the tour } x,\\
0,& \mbox {otherwise.}
\end{cases}
\end{gather*}

The polytope
$$\PYR(n) = \conv \{x^{v}\ | \ x \in PT_{n}\}$$
is called the pyramidal tours polytope.

We use a special encoding to represent the pyramidal tours. With each pyramidal tour $x \in PT_{n}$ we associate a $0/1$ vector $x^{c} \in \R^{n-3}$ by the following rule:
\begin{gather*}
\forall i\ (3 \leq i \leq n-1)\\
x^{c}_{i} =
\begin{cases}
1,& \mbox {if a vertex } i \mbox { is contained in the tour } x\\
& \mbox{in increasing order},\\
0,& \mbox {otherwise.}
\end{cases}
\end{gather*}

Since we consider a symmetric traveling salesperson problem, all tours are undirected, so the increasing and decreasing orders are a matter of agreement. 
Let the edge $(1,2)$ define the increasing order. 
Therefore, we begin the numbering of vertices in the encoding from $3$, since it is the first potential branching on the pyramidal tour. 
Thus, the total number of pyramidal tours in $K_{n}$ (vertices of the polytope $\PYR(n)$) is $2^{n-3}$.
An example of a pyramidal tour and its corresponding encoding is shown in Fig. \ref{Fig_example_pyramidal_tour}.

\begin{figure} [h]
	\centering
	\begin{tikzpicture}[scale=1.0]
	\begin{scope}[every node/.style={circle,thick,draw,inner sep=3pt}]
	\node (A) at (0,0) {1};
	\node (B) at (1,0) {2};
	\node (C) at (2,0) {3};
	\node (D) at (3,0) {4};
	\node (E) at (4,0) {5};
	\node (F) at (5,0) {6};
	\node (G) at (6,0) {7};
	\node (H) at (7,0) {8};
	\end{scope}
	\draw [line width=0.25mm] (A) edge (B);
	\draw [line width=0.3mm] (B) edge [bend left=50] (D);
	\draw [line width=0.25mm] (D) edge (E);
	\draw [line width=0.3mm] (E) edge [bend left=50] (G);
	\draw [line width=0.25mm] (G) edge (H);
	\draw [line width=0.3mm] (H) edge [bend left=50] (F);
	\draw [line width=0.3mm] (F) edge [bend left=50] (C);
	\draw [line width=0.3mm] (C) edge [bend left=50] (A);
	\end{tikzpicture}
	\caption {An example of a pyramidal tour $\left\langle 0,1,1,0,1 \right\rangle$}
	\label{Fig_example_pyramidal_tour}
\end{figure}
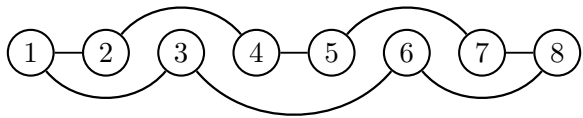

\section {The skeleton of the pyramidal tours polytope}
\indent 

\begin{lemma} \label {lemma_adjacency}
	Vertices $x^{v}$ and $y^{v}$ of the polytope $\PYR(n)$ are not adjacent if and only if it is possible to compose another pyramidal tour $z$ of the edges of tours $x$ and $y$.
\end{lemma}

\begin{proof}
	
	\textit{Necessity} follows directly from the assumption of nonadjacency of $x^{v}$ and $y^{v}$.
	Indeed, the nonadjacency of $x^{v}$ and $y^{v}$ means that the segment joining them contains some convex combination of the remaining vertices of the polytope $\PYR (n)$:
	\begin{gather*}
	\alpha x^{v} + \beta y^{v} = \sum {\gamma_{z} z^{v}},\\
	\alpha + \beta = \sum {\gamma_{z}} = 1,\\
	\alpha \geq 0, \ \beta \geq 0, \ \gamma_{z} \geq 0.
	\end{gather*}
	In this convex combination, at least one point $z^{v}$ has a positive coefficient $\gamma_{z}$.
	Therefore, since all vertices of the polytope $\PYR (n)$ are Boolean, it follows that the pyramidal tour $z$ that corresponds to the vertex $z^{v}$ is composed of the edges of tours $x$ and $y$.
	
	Let us prove \textit{sufficiency}.
	Suppose that it is possible to compose a new pyramidal tour $z$ of the edges of tours $x$ and $y$.
	We consider a multigraph $G = x \cup y$ in which the edges belonging simultaneously to $x$ and $y$ are included twice.
	Then the degree of each vertex of the graph $G$ is equal to $4$.
	We construct a tour $w = G \backslash z$. Let us prove that $w$ is also a pyramidal tour.
	By construction, the degree of each vertex of $w$ is $2$.
	Thus, $w$ consists of one or more cycles.
	Moreover, for each $k$ ($1 < k < n$) among the remaining two edges incident to the vertex $k$, one has the form $(i,k)$, where $i < k$, and the other -- $(k,j)$, where $k < j$.
	If $w$ consists of more than one cycle, then there is a cycle that does not contain the vertex $n$.
	Let $k$ be the vertex with the largest number in this cycle, then both vertices adjacent to $k$ in the graph $w$ have smaller numbers.
	Contradiction, the Lemma \ref{lemma_adjacency} is proved.
\end{proof}

It should be noted that for a general traveling salesperson problem a similar statement: if it is possible to compose a third Hamiltonian cycle $z$ from the edges of the Hamiltonian cycles $x$ and $y$, then the remaining edges also form a Hamiltonian cycle, is false.
As an example, we can consider the Hamiltonian cycles shown in Fig. \ref{Fig_general_TSP_example}.

\begin{figure} [h]
	\centering
	\begin{tikzpicture}[scale=1.0]
	\begin{scope}[every node/.style={circle,thick,draw}]
	\node (A) at (0,0) {1};
	\node (B) at (1,1) {2};
	\node (C) at (2.5,1) {3};
	\node (D) at (3.5,0) {4};
	\node (E) at (2.5,-1) {5};
	\node (F) at (1,-1) {6};
	\end{scope}
	\draw [line width=0.3mm] (A) edge (D);
	\draw [line width=0.3mm] (D) edge (E);
	\draw [line width=0.3mm] (E) edge (C);
	\draw [line width=0.3mm] (C) edge (B);
	\draw [line width=0.3mm] (B) edge (F);
	\draw [line width=0.3mm] (F) edge (A);
	\draw (1.75, -1.7) node{\textit{x}};
	\end{tikzpicture}
	\hspace*{6mm}
	\begin{tikzpicture}[scale=1.0]
	\begin{scope}[every node/.style={circle,thick,draw}]
	\node (A) at (0,0) {1};
	\node (B) at (1,1) {2};
	\node (C) at (2.5,1) {3};
	\node (D) at (3.5,0) {4};
	\node (E) at (2.5,-1) {5};
	\node (F) at (1,-1) {6};
	\end{scope}
	\draw [line width=0.3mm] (A) edge (B);
	\draw [line width=0.3mm] (B) edge (F);
	\draw [line width=0.3mm] (F) edge (D);
	\draw [line width=0.3mm] (D) edge (C);
	\draw [line width=0.3mm] (C) edge (E);
	\draw [line width=0.3mm] (E) edge (A);
	\draw (1.75, -1.7) node{\textit{y}};
	\end{tikzpicture}
	\\
	\begin{tikzpicture}[scale=1.0]
	\begin{scope}[every node/.style={circle,thick,draw}]
	\node (A) at (0,0) {1};
	\node (B) at (1,1) {2};
	\node (C) at (2.5,1) {3};
	\node (D) at (3.5,0) {4};
	\node (E) at (2.5,-1) {5};
	\node (F) at (1,-1) {6};
	\end{scope}
	\draw [line width=0.3mm] (A) edge (D);
	\draw [line width=0.3mm] (D) edge (F);
	\draw [line width=0.3mm] (F) edge (B);
	\draw [line width=0.3mm] (B) edge (C);
	\draw [line width=0.3mm] (C) edge (E);
	\draw [line width=0.3mm] (E) edge (A);
	\draw (1.75, -1.7) node{\textit{z}};
	\end{tikzpicture}
	\hspace*{6mm}
	\begin{tikzpicture}[scale=1.0]
	\begin{scope}[every node/.style={circle,thick,draw}]
	\node (A) at (0,0) {1};
	\node (B) at (1,1) {2};
	\node (C) at (2.5,1) {3};
	\node (D) at (3.5,0) {4};
	\node (E) at (2.5,-1) {5};
	\node (F) at (1,-1) {6};
	\end{scope}
	\draw [line width=0.3mm] (A) edge (B);
	\draw [line width=0.3mm] (B) edge (F);
	\draw [line width=0.3mm] (F) edge (A);
	\draw [line width=0.3mm] (C) edge (D);
	\draw [line width=0.3mm] (D) edge (E);
	\draw [line width=0.3mm] (E) edge (C);
	\draw (1.75, -1.7) node{\textit{w}};
	\end{tikzpicture}
	\caption {An example of $ w = (x \cup y) \backslash z $ that is not a Hamiltonian cycle}
	\label {Fig_general_TSP_example}
\end{figure}
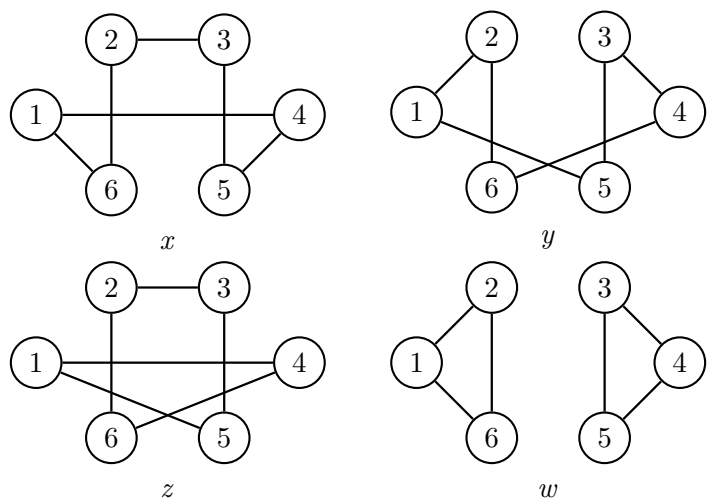

\begin{theorem} \label {theorem_adjacency}
	Vertices $x^{v}$ and $y^{v}$ of the polytope $\PYR(n)$ are not adjacent if and only if one of the following two sufficient conditions is satisfied:
	\begin{enumerate}[1)]
		\item there exists such $k$ $(3 < k < n-2)$ that
		\begin{equation} \label{first_sufficient_equality}
		x^{c}_{k} = y^{c}_{k} \neq x^{c}_{k+1} = y^{c}_{k+1},
		\end{equation}
		and there exist such $i$ $(i < k)$ and $j$ $(j > k+1)$ that
		\begin{equation} \label{first_sufficient_inequality}
		x^{c}_{i} \neq y^{c}_{i},\ x^{c}_{j} \neq y^{c}_{j};
		\end{equation}
		\item there exists such $k$ $(3 \leq k < n-2)$ that
		\begin{equation} \label {second_sufficient_inequality}
		x^{c}_{k} = y^{c}_{k+1} \neq x^{c}_{k+1} = y^{c}_{k},
		\end{equation}
		and there exists such $j$ $(j > k+1)$ that
		\begin{equation} \label {second_sufficient_equality}
		x^{c}_{j} = y^{c}_{j}.
		\end{equation}
	\end{enumerate}
\end{theorem}

\begin{proof}
	\textit{Sufficiency}. Let the first sufficient condition be satisfied. We construct a pyramidal tour $z$ by the following rule:
	\begin{gather*}
	z^{c}_{i} = 
	\begin{cases}
	x^{c}_{i},& \mbox {if } i \leq k, \\
	y^{c}_{i},& \mbox {if } i > k. 
	\end{cases}
	\end{gather*}
	By construction, the tour $z$ consists entirely of the edges of the tours $x$ and $y$.
	We make a jump between the edges of $x$ and $y$ by the condition (\ref {first_sufficient_equality}).
	In this case, $z$ is different from the tours $x$ and $y$ by (\ref {first_sufficient_inequality}):
	$$z^{c}_{i} \neq y^{c}_{i},\ z^{c}_{j} \neq x^{c}_{j}.$$
	By Lemma \ref {lemma_adjacency}, the vertices $x^{v}$ and $y^{v}$ of the polytope $\PYR (n) $ are not adjacent.
	
	An example of the first sufficient condition for the tours $x^{c} = \left\langle 1,1,0,1,1 \right\rangle$, $y^{c} = \left\langle 0,1,0,0,1 \right\rangle$ and $k = 4$ is shown in Fig. \ref{Fig_first_sufficient_condition}. 
	
	\begin{figure} [h]
		\centering
		\begin{tikzpicture}[scale=1.0]
		\begin{scope}[every node/.style={circle,thick,draw,inner sep=3pt}]
		\node (A) at (0,0) {1};
		\node (B) at (1,0) {2};
		\node (C) at (2,0) {3};
		\node (D) at (3,0) {4};
		\node (E) at (4,0) {5};
		\node (F) at (5,0) {6};
		\node (G) at (6,0) {7};
		\node (H) at (7,0) {8};
		\end{scope}
		\draw [line width=0.25mm] (A) edge (B);
		\draw [line width=0.25mm] (B) edge (C);
		\draw [line width=0.25mm] (C) edge (D);
		\draw [line width=0.3mm] (D) edge [bend left=50] (F);
		\draw [line width=0.25mm] (F) edge (G);
		\draw [line width=0.25mm] (G) edge (H);
		\draw [line width=0.3mm] (H) edge [bend left=45] (E);
		\draw [line width=0.3mm] (E) edge [bend left=40] (A);
		\draw (-1, 0) node{\textit{x}};
		\end{tikzpicture}
		\\
		\begin{tikzpicture}[scale=1.0]
		\begin{scope}[every node/.style={circle,thick,draw,inner sep=3pt}]
		\node (A) at (0,0) {1};
		\node (B) at (1,0) {2};
		\node (C) at (2,0) {3};
		\node (D) at (3,0) {4};
		\node (E) at (4,0) {5};
		\node (F) at (5,0) {6};
		\node (G) at (6,0) {7};
		\node (H) at (7,0) {8};
		\end{scope}
		\draw [line width=0.25mm] (A) edge (B);
		\draw [line width=0.3mm] (B) edge [bend left=50] (D);
		\draw [line width=0.3mm] (D) edge [bend left=45] (G);
		\draw [line width=0.25mm] (G) edge (H);
		\draw [line width=0.3mm] (H) edge [bend left=50] (F);
		\draw [line width=0.25mm] (F) edge (E);
		\draw [line width=0.3mm] (E) edge [bend left=50] (C);
		\draw [line width=0.3mm] (C) edge [bend left=50] (A);
		\draw (-1, 0) node{\textit{y}};
		\end{tikzpicture}
		\\
		\begin{tikzpicture}[scale=1.0]
		\begin{scope}[every node/.style={circle,thick,draw,inner sep=3pt}]
		\node (A) at (0,0) {1};
		\node (B) at (1,0) {2};
		\node (C) at (2,0) {3};
		\node (D) at (3,0) {4};
		\node (E) at (4,0) {5};
		\node (F) at (5,0) {6};
		\node (G) at (6,0) {7};
		\node (H) at (7,0) {8};
		\end{scope}
		\draw [line width=0.25mm] (A) edge (B);
		\draw [line width=0.25mm] (B) edge (C);
		\draw [line width=0.25mm] (C) edge (D);
		\draw [line width=0.3mm] (D) edge [bend left=45] (G);
		\draw [line width=0.25mm] (G) edge (H);
		\draw [line width=0.3mm] (H) edge [bend left=50] (F);		
		\draw [line width=0.25mm] (F) edge (E);
		\draw [line width=0.3mm] (E) edge [bend left=40] (A);
		\draw (-1, 0) node{\textit{z}};
		\end{tikzpicture}
		\caption{An example of the first sufficient condition}
		\label{Fig_first_sufficient_condition}
	\end{figure}
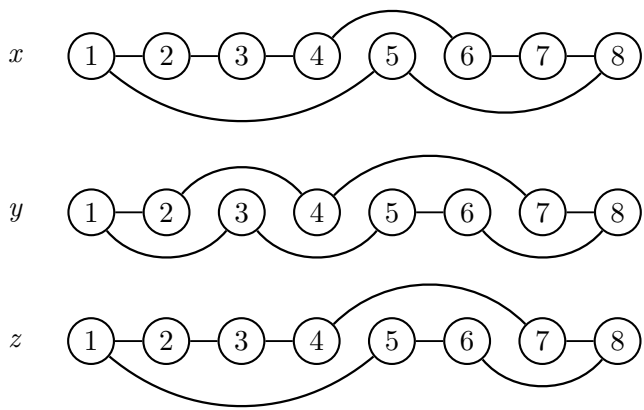
	
	Now let the second sufficient condition be satisfied. We construct a pyramidal tour $z$ by the following rule:
	\begin{gather*}
	z^{c}_{i} = 
	\begin{cases}
	x^{c}_{i},& \mbox {if } i \leq k, \\
	1 - y^{c}_{i},& \mbox {if } i > k. 
	\end{cases}
	\end{gather*}
	
	The difference from the first sufficient condition is that after the jump between tours by (\ref {second_sufficient_inequality}), the edges of $y$ in the decreasing order become the edges of $z$ in the increasing order and vice versa.
	This is possible since we consider the symmetric traveling salesperson problem.
	The constructed tour $z$ is different from the tours $x$ and $y$ due to (\ref {second_sufficient_equality}):
	$$z^{c}_{j} \neq x^{c}_{j} = y^{c}_{j}.$$
	By Lemma \ref {lemma_adjacency}, the vertices $x^{v}$ and $y^{v}$ of the polytope $\PYR (n) $ are not adjacent.
	
	An example of the second sufficient condition for the tours $x^{c} = \left\langle 0,1,0,0,0 \right\rangle$, $y^{c} = \left\langle 1,0,1,0,0 \right\rangle$ and $k = 4$ is shown in Fig. \ref{Fig_second_sufficient_condition}.
	
	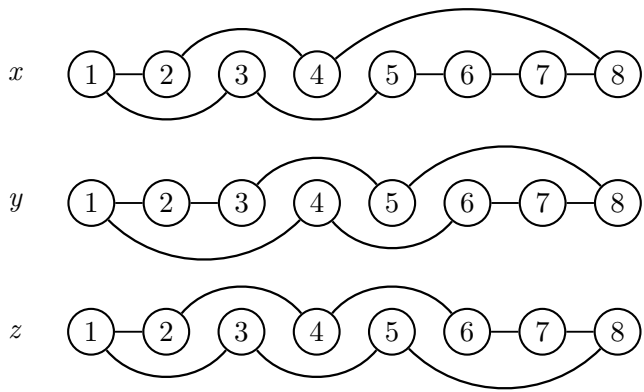
\begin{figure} [h]
		\centering
		\begin{tikzpicture}[scale=1.0]
		\begin{scope}[every node/.style={circle,thick,draw,inner sep=3pt}]
		\node (A) at (0,0) {1};
		\node (B) at (1,0) {2};
		\node (C) at (2,0) {3};
		\node (D) at (3,0) {4};
		\node (E) at (4,0) {5};
		\node (F) at (5,0) {6};
		\node (G) at (6,0) {7};
		\node (H) at (7,0) {8};
		\end{scope}
		\draw [line width=0.25mm] (A) edge (B);
		\draw [line width=0.3mm] (B) edge [bend left=50] (D);
		\draw [line width=0.3mm] (D) edge [bend left=40] (H);
		\draw [line width=0.25mm] (H) edge (G);
		\draw [line width=0.25mm] (G) edge (F);
		\draw [line width=0.25mm] (F) edge (E);
		\draw [line width=0.3mm] (E) edge [bend left=50] (C);		
		\draw [line width=0.3mm] (C) edge [bend left=50] (A);		
		\draw (-1, 0) node{\textit{x}};
		\end{tikzpicture}
		\\
		\begin{tikzpicture}[scale=1.0]
		\begin{scope}[every node/.style={circle,thick,draw,inner sep=3pt}]
		\node (A) at (0,0) {1};
		\node (B) at (1,0) {2};
		\node (C) at (2,0) {3};
		\node (D) at (3,0) {4};
		\node (E) at (4,0) {5};
		\node (F) at (5,0) {6};
		\node (G) at (6,0) {7};
		\node (H) at (7,0) {8};
		\end{scope}
		\draw [line width=0.25mm] (A) edge (B);
		\draw [line width=0.25mm] (B) edge (C);		
		\draw [line width=0.3mm] (C) edge [bend left=50] (E);		
		\draw [line width=0.3mm] (E) edge [bend left=45] (H);
		\draw [line width=0.25mm] (H) edge (G);		
		\draw [line width=0.25mm] (G) edge (F);		
		\draw [line width=0.3mm] (F) edge [bend left=50] (D);		
		\draw [line width=0.3mm] (D) edge [bend left=45] (A);
		\draw (-1, 0) node{\textit{y}};
		\end{tikzpicture}
		\\
		\begin{tikzpicture}[scale=1.0]
		\begin{scope}[every node/.style={circle,thick,draw,inner sep=3pt}]
		\node (A) at (0,0) {1};
		\node (B) at (1,0) {2};
		\node (C) at (2,0) {3};
		\node (D) at (3,0) {4};
		\node (E) at (4,0) {5};
		\node (F) at (5,0) {6};
		\node (G) at (6,0) {7};
		\node (H) at (7,0) {8};
		\end{scope}
		\draw [line width=0.25mm] (A) edge (B);
		\draw [line width=0.3mm] (B) edge [bend left=50] (D);
		\draw [line width=0.3mm] (D) edge [bend left=50] (F);
		\draw [line width=0.25mm] (F) edge (G);
		\draw [line width=0.25mm] (G) edge (H);
		\draw [line width=0.3mm] (H) edge [bend left=45] (E);
		\draw [line width=0.3mm] (E) edge [bend left=50] (C);
		\draw [line width=0.3mm] (C) edge [bend left=50] (A);
		\draw (-1, 0) node{\textit{z}};
		\end{tikzpicture}
		\caption{An example of the second sufficient condition}
		\label{Fig_second_sufficient_condition}
	\end{figure}
	
	\textit {Necessity.}
	Let $z$ be a pyramidal tour consisting of edges of $x$ and $y$, but different from them.
	Suppose that the tour $z$ enters the vertex $n$ along the edge $(i,n)$ of the tour $x$, while leaves by some edge of the tour $y$, so at the vertex $n$ there is a jump between two tours.
	Then, to ensure the pyramidality, the tour $z$ must visit in the decreasing order all the vertices $j$ ($i < j < n$) that were omitted earlier in the tour $x$ and bypass the already visited vertex $i$:
	$$x^{c}_{j} = y^{c}_{j},\ x^{c}_{i} = y^{c}_{i} \mbox { или } x^{c}_{j} = 1 - y^{c}_{j},\ x^{c}_{i} = 1 - y^{c}_{i}.$$
	But in this case the edge $(i, n)$ also belongs to the tour $y$.
	Thus, we can assume that the tour $z$ enters the vertex $n$ and leaves it along the edges of one tour.
	Let it be the tour $y$ (otherwise we can replace $y$ with $x$ here and below).
	In addition, if some edge of $z$ belongs to both tours $x$ and $y$, then we assume that it is an edge of $y$.
	
	Note that $z$ contains at least two edges from $x \backslash y$, the edges that are unique for the tour $x$: one edge for the increasing order and one for the decreasing order. Otherwise, $z$ coincides with $y$.
	
	We choose a pair of unique edges of $x$ that are in $z$: $(i,k)$ for the increasing order, $(s,q)$ for the decreasing order, with the largest vertex numbers $k$ and $q$.
	The case $k = q$ is excluded, otherwise the vertex with the number $k$ is visited twice in the tour $z$. We also note that by construction $k < n$.
	
	Without loss of generality we assume $x^{c}_{i} = x^{c}_{k} = 1 $. The case $x^{c}_{i} = x^{c}_{k} = 0 $ is treated similarly.
	By construction, to ensure the pyramidal property, we have:
	\begin{gather}
	\nonumber
	x^{c}_{i} = x^{c}_{k} = z^{c}_{i} = z^{c}_{k} = 1,\\
	\forall j\ (i < j < k):\ z^{c}_{j} = x^{c}_{j} = 0. \label {ik_x}
	\end{gather}
	Thus, on the fragment $\left[i, k \right]$ the tour $z$ has the same edges as the tour $x$.
	
	Since the edge $(i,k)$ has the largest number $k$ among the unique edges of $x$ that are in $z$, the tour $z$ at the vertex $k$ makes a jump from the edges of $x$ to the edges of $y$, passes through the vertex $n$ and moves along the edges of $y$ to the vertex $q$ where the next unique edge of $x$ begins. Let us consider two cases.
	
	\begin{enumerate}
		\item Let $y^{c}_{k} = 1$. Then on the fragment $\left[q,n\right]$ the tour $z$ inherits the increasing and decreasing orders of the tour $y$:
		\begin {equation} \label {qn_y}
		\forall j \geq q:\ z^{c}_{j} = y^{c}_{j}.
		\end {equation}
		By combining (\ref {ik_x}) and (\ref{qn_y}) with the inequality $q < k$, we obtain a block of the required form (\ref{first_sufficient_equality}):
		\begin {gather*}
		x^{c}_{k} = z^{c}_{k} = y^{c}_{k} = 1,\\
		x^{c}_{k-1} = z^{c}_{k-1} = y^{c}_{k-1} = 0.
		\end {gather*}
		Note that the case $i = k-1$ is excluded, otherwise the edge $(i,k)$ belongs to the tour $y$ and cannot be unique for $x$.
		
		It remains to verify that the condition (\ref {first_sufficient_inequality}) is satisfied.
		We suppose that for any $t$ ($t < k-1$): $ x^{c}_{t} = y^{c}_{t}$. Then, by (\ref{qn_y}), the tours $z$ and $y$ coincide, a contradiction.
		
		Now we suppose that for any $j$ ($j > k$): $x^{c}_{j} = y^{c}_{j}$. Then on the fragment $\left[i,n\right]$ the tour $z$ completely consists of the edges of $x$.
		We rename the tour $y$ as $x$, find the unique edge $(\tilde {i}, \tilde {k})$ of it, appearing in $z$ with the largest number $\tilde {k}$, and repeat the argument.
		Note that by construction $\tilde {k} < i$, and hence there can be only a finite number of such operations. At some step we get a contradiction.
		
		Thus, for the case $y^{c}_{k} = 1$ the first sufficient condition is satisfied.
		
		\item Let $y^{c}_{k} = 0$. Then on the fragment $\left[q,n\right]$ the tour $z$ inverts the increasing and decreasing orders of the tour $y$:
		\begin{equation} \label {qn_y_inv}
		\forall j \geq q:\ z^{c}_{j} = 1 - y^{c}_{j}.
		\end{equation}
		Again, by combining (\ref {ik_x}) and (\ref {qn_y_inv}) with the inequality $q < k$, we obtain a block of the required form (\ref {second_sufficient_inequality}):
		\begin {gather*}
		x^{c}_{k} = z^{c}_{k} = y^{c}_{k-1} = 1,\\
		x^{c}_{k-1} = z^{c}_{k-1} = y^{c}_{k} = 0.
		\end {gather*}
		
		It remains to verify that the condition (\ref {second_sufficient_equality}) is satisfied.
		We suppose that for any $j$ ($j > k$): $x^{c}_{j} = 1 - y^{c}_{j}$.
		Again, on the fragment $\left[i,n\right]$ the tour $z$ consists of the edges of $x$.
		We rename the tours $x$ and $y$, repeating the arguments in the previous step. After a finite number of such operations we obtain a contradiction.
		
		Therefore, for the case $y^{c}_{k} = 0$ the second sufficient condition is satisfied.
		\end {enumerate}
		
		Thus, the combination of two sufficient conditions is a necessary condition. Theorem \ref{theorem_adjacency} is proved.
	\end{proof}

	Theorem \ref{theorem_adjacency} provides an efficient criterion for verifying the adjacency of the vertices of the pyramidal tours polytope (Algorithm \ref{adjacency_alg}). This fact fundamentally distinguishes it from the general traveling salesperson polytope $\TSP (n)$ for which the similar problem is NP-complete (Proposition \ref{TSP_adjacency_NPC}).
	
	\begin{algorithm}[h]
		\caption{($\PYR(n)$ vertices adjacency test)}\label{adjacency_alg}
		\begin{algorithmic}[1]
			\Procedure{Adjacency}{$x,y$}
			
			\State $i\gets 3$ \Comment{Verifying first sufficient condition}
			
			\Repeat
			\State $i\gets i+1$
			\Until{$x^{c}_{i} \neq y^{c}_{i}$}
			
			\State $k\gets i+1$
			\Repeat
			\State $k\gets k+1$
			\Until{$x^{c}_{k} = y^{c}_{k} \neq x^{c}_{k+1} = y^{c}_{k+1}$} \Comment {Searching for a block (\ref {first_sufficient_equality})}
			
			\For{$j\gets k+2, n-1$}
			\If{$x^{c}_{j} \neq y^{c}_{j}$} \Comment{First sufficient condition is satisfied}
			\State \Return Vertices $x^{v}$ and $y^{v}$ are nonadjacent
			\EndIf
			\EndFor
			
			\State $k\gets 3$ \Comment{Verifying second sufficient condition}
			\Repeat
			\State $k\gets k+1$
			\Until{$x^{c}_{k} = y^{c}_{k+1} \neq x^{c}_{k+1} = y^{c}_{k}$} \Comment {Searching for a block (\ref {second_sufficient_inequality})}
			
			\For{$j\gets k+2, n-1$}
			\If{$x^{c}_{j} = y^{c}_{j}$} \Comment{Second sufficient condition is satisfied}
			\State \Return  Vertices $x^{v}$ and $y^{v}$ are nonadjacent
			\EndIf
			\EndFor \Comment{A necessary condition for nonadjacency is not satisfied}
			
			\State\Return Vertices $x^{v}$ and $y^{v}$ are adjacent
			\EndProcedure
		\end{algorithmic}
	\end{algorithm}

	\begin {theorem} \label {linear_adjacency_check}
	The question whether two vertices of $\PYR(n)$ are adjacent can be verified in linear time $O(n)$.
	\end {theorem}
	
	\begin{proof}
		Note that Algorithm \ref{adjacency_alg} in the worst case requires a double passing along the coordinates of the vectors $x^{c}$ and $y^{c}$ to verify two sufficient conditions for nonadjacency.
		Theorem \ref {linear_adjacency_check} is proved.
	\end{proof}

	\section{Diameter and clique number}
	\indent
	
	On the basis of the nonadjacency criterion of Theorem \ref{theorem_adjacency} we examine the diameter and the clique number of the skeleton of the pyramidal tours polytope.
	
	\begin{theorem}\label {diameter_theorem}
		The diameter of $\PYR (n)$ skeleton equals 2 for all $n \geq 6$.
	\end{theorem}
	
	\begin{proof}
		First we note that for $n \leq 5$ the necessary condition of Theorem \ref{theorem_adjacency} is not satisfied, and all vertices of the skeleton are pairwise adjacent.
		Starting with $n = 6$, polytope $\PYR (n)$ has pairs of vertices that satisfy at least one of the sufficient nonadjacency conditions. For example, vertices $\left\langle 1,0,0 \right\rangle$ and $\left\langle 0,1,0 \right\rangle$ are not adjacent. 
		
		It remains to note that by Theorem \ref {theorem_adjacency} two vertices with the codes 
		$$\left\langle 1,1,1, \ldots, 1 \right\rangle \mbox { and } \left\langle 0,0,0, \ldots, 0 \right\rangle$$ 
		are adjacent to all vertices of the pyramidal tours polytope.
		Theorem \ref{diameter_theorem} is proved.
	\end{proof}
	
	Therefore, if we confine ourselves to considering only pyramidal tours for the traveling salesperson problem, the conjecture of Gr\"{o}tschel and Padberg on the diameter of the skeleton \cite{Grotschel} is correct. 
	
	\begin{theorem}\label{clique_theorem}
		The clique number of $\PYR(n)$ skeleton is quadratic in the parameter $n$:
		\begin{equation} \label{quadratic_clique}
		\omega (\PYR(n)) = \Theta (n^{2}).
		\end{equation}
	\end{theorem}
	
	\begin{proof}
		We estimate the clique number of the skeleton from above. Let $Y_{v}$ be a set of pairwise adjacent vertices of $\PYR(n)$, and $Y$ be the set of corresponding pyramidal tours.
		
		We choose $k$ $(3 \leq k \leq n-2)$.
		Pyramidal tour $y \in Y$ is called unique with respect to $k$, if
		\begin {itemize}
		\item $y^{c}_{k} \neq y^{c}_{k+1}$;
		\item for all $z \in Y\backslash {y}$: $z^{c}_{k} \neq y^{c}_{k}$ or $z^{c}_{k+1} \neq y^{c}_{k+1}$.
		\end {itemize}
		Thus, the block of code in $y^{c}$ on the coordinates $\left[k, k + 1\right]$ has the form $\left\langle 1,0 \right\rangle$ or $\left\langle 0,1 \right\rangle$, and there is no such block on these coordinates in any tour of $Y$.
		
		We construct the set $W$, excluding from $Y$ all unique pyramidal tours.
		Note that the number of excluded pyramidal tours does not exceed $2(n-4)$.
		We consider a tour $x \in W$. Suppose that for some $k$ $(3 < k < n-2)$: $x^{c}_{k} \neq x^{c}_{k+1}$. 
		By construction of the set $W$, there is a tour $y \in W$ such that $x^{c}_{k} = y^{c}_{k}$ and $x^{c}_{k+1} = y^{c}_{k+1}$.
		Since the vertices $x^{v}$ and $y^{v}$ of the polytope $\PYR(n)$ are adjacent, by Theorem \ref{theorem_adjacency} either their left fragments of the code with respect to $k$ coincide
		$$\forall i \ (3 \leq i < k):\ x^{c}_{i} = y^{c}_{i},$$
		or the right fragments of the code coincide
		$$\forall j \ (k+1 < j \leq n-1):\ x^{c}_{j} = y^{c}_{j}.$$
		Otherwise, the vertices $x^{v}$ and $y^{v}$ are not adjacent by the first sufficient condition.
		
		Note that for any tours with a common block, the coincident fragments are on the same side of the common block.
		Indeed, suppose that three tours $x,y,z \in W$ have a common block of the form $\left\langle 1,0 \right\rangle$ on the position $\left[k, k + 1\right]$, and at the same time
		\begin{gather*}
		\forall i \ (3 \leq i < k):\ x^{c}_{i} = y^{c}_{i},\\
		\forall j \ (k+1 < j \leq n-1):\ x^{c}_{j} = z^{c}_{j}.
		\end{gather*}
		Then, since the vertices $y^{v}$ and $z^{v}$ are adjacent, either their left fragments of the code with respect to the block $\left[k,k + 1\right]$ coincide (in this case $x = z$), or the right fragments of the code coincide ($x = y$). We have a contradiction.
		
		Thus, for each pyramidal tour from $W$, all the blocks of the form $\left\langle 1,0 \right\rangle$ and $\left\langle 0,1 \right\rangle$ can be divided into two classes:  those where the right fragments coincide, and those where the left fragments coincide.
		With each pyramidal tour $x \in W$ we associate a vector $x^{\rightarrow}$ by the following rule:
		\begin{gather*}
		x^{\rightarrow}_{k} =
		\begin{cases}
		(\rightarrow), &\mbox {if } x^{c}_{k} \neq x^{c}_{k+1}, \mbox { and with respect to the block } \left[ k,k+1 \right]\\
		&\mbox {the right fragments coincide,}\\
		(\leftarrow), &\mbox {if } x^{c}_{k} \neq x^{c}_{k+1}, \mbox { and with respect to the block } \left[ k,k+1 \right]\\
		&\mbox {the left fragments coincide,}\\
		(-), &\mbox {if } x^{c}_{k} = x^{c}_{k+1}.
		\end{cases}
		\end{gather*}
		
		We note that coinciding fragments cannot overlap on one tour $x \in W$. Indeed, if for some $k,s$: $x^{\rightarrow}_{k} = (\leftarrow)$, and $x^{\rightarrow}_{s} = (\rightarrow)$, then $k < s$.
		Suppose the contrary. We consider a tour $y \in W$ such that
		\begin {gather*}
		y^{c}_{k} = x^{c}_{k},\ y^{c}_{k+1} = x^{c}_{k+1},\\
		\forall i \ (3 \leq i < k):\ x^{c}_{i} = y^{c}_{i}.
		\end {gather*}
		But, by assumption, $s \leq k$, and so the blocks $\left[s, s + 1 \right]$ of the tours $x$ and $y$ also coincide:
		\begin {gather*}
		y^{c}_{s} = x^{c}_{s},\ y^{c}_{s+1} = x^{c}_{s+1},\\
		\forall j \ (s+1 < j \leq n-1):\ x^{c}_{j} = y^{c}_{j}.
		\end {gather*}
		The tours $x$ and $y$ are equal to each other, a contradiction.
		
		We consider a tour $x \in W$ and choose the largest value $k$ for which $x^{\rightarrow}_{k} = (\leftarrow)$ and the smallest value $s$ for which $x^{\rightarrow}_{s} = (\rightarrow)$.
		If the tour $x$ does not contain blocks $(\leftarrow)$ or $(\rightarrow)$, we denote the corresponding element by the symbol $\emptyset$.
		Note that the values of the coordinates $x^{c}_{k+1}$ and $x^{c}_{s}$ coincide.
		Otherwise, there are blocks of the form $\left\langle 1,0 \right\rangle$ or $\left\langle 0,1 \right\rangle$ between $k + 1$ and $s$.
		
		Thus, with each tour $x \in W$ we can associate a triple $\left(k,s,x^{c}_{k + 1} = x^{c}_{s}\right)$ that uniquely defines $x$ among tours of $W$.
		Since $k,s \in \{4,5,\ldots, n-3,\emptyset\}$, the total number of triples $\left (k,s,0/1 \right)$ does not exceed
		$$|W| \leq 2 (n-5)(n-6).$$
		Taking into account the previously excluded unique tours, we obtain the desired upper bound:
		$$\omega (\PYR(n))  = O (n^{2}).$$
		
		Now we estimate the clique number of the skeleton of the polytope $\PYR(n)$ from below.
		Let
		$$m = \left\lfloor \frac {n-3}{4}\right\rfloor.$$
		We consider a set of pyramidal tours $Z$.
		With each pair $q,s$, where $0 \leq q,s \leq m$, we associate a pyramidal tour $x \in Z$ according to the following rules:
		\begin {itemize}
		\item $\forall i$ $(1 \leq i \leq q)$: $x^{c}_{2i+1} = 1$, $x^{c}_{2i+2} = 0$;
		\item $\forall j$ $(1 \leq j \leq s)$: $x^{c}_{4m-2j+3} = 0$, $x^{c}_{4m-2j+4} = 1$;
		\item $\forall k \geq 4m + 3$: $x^{c}_{k} = 1$;
		\item all the remaining coordinates of $x^{c}$ are equal to zero.
		\end {itemize}
		The total number of such tours is $(m + 1)^{2}$. Here is an example of the set $Z$ for $n= 12$ ($m = 2$):
		\begin {gather*}
		(0,0) = \left\langle 0,0, 0,0, 0,0, 0,0, 1 \right\rangle,\\
		(0,1) = \left\langle 0,0, 0,0, 0,0, 0,1, 1 \right\rangle,\\
		(0,2) = \left\langle 0,0, 0,0, 0,1, 0,1, 1 \right\rangle,\\ 
		(1,0) = \left\langle 1,0, 0,0, 0,0, 0,0, 1 \right\rangle,\\ 
		(1,1) = \left\langle 1,0, 0,0, 0,0, 0,1, 1 \right\rangle,\\
		(1,2) = \left\langle 1,0, 0,0, 0,1, 0,1, 1 \right\rangle,\\
		(2,0) = \left\langle 1,0, 1,0, 0,0, 0,0, 1 \right\rangle,\\ 
		(2,1) = \left\langle 1,0, 1,0, 0,0, 0,1, 1 \right\rangle,\\ 
		(2,2) = \left\langle 1,0, 1,0, 0,1, 0,1, 1 \right\rangle.
		\end {gather*}
		
		It remains to verify that, by construction, if for some pair of tours $x,y \in Z$:
		$x^{c}_{k} = y^{c}_{k} \neq x^{c}_{k+1} = y^{c}_{k+1}$, then
		$$\forall i \ (3 \leq i < k):\ x^{c}_{i} = y^{c}_{i},$$
		if $k \leq 2m+2$, and
		$$\forall j \ (k+1 < j \leq n-1):\ x^{c}_{j} = y^{c}_{j},$$
		if $k > 2m+2$.
		Consequently, the first sufficient condition of nonadjacency is not satisfied.
		And there is no pair of tours $x,y \in Z$ such that $x^{c}_{k} = y^{c}_{k + 1} \neq x^{c}_{k + 1} = y^{c}_{k}$.
		Thus, the second sufficient condition of nonadjacency is also not satisfied, and all vertices of the polytope $\PYR (n)$ that correspond to the pyramidal tours of the set $Z$ are pairwise adjacent.
		We obtain the lower bound
		$$\omega (\PYR(n)) \geq \left( \left\lfloor \frac {n-3}{4}\right\rfloor + 1\right)^2,$$
		and the quadratic asymptotically exact estimate (\ref{quadratic_clique}) from the statement of the theorem.
		Theorem \ref{clique_theorem} is proved.
	\end{proof}
	
	Thus, the clique number of $\PYR (n)$ skeleton differs in principle from the exponential clique number of the skeleton of the general traveling salesperson polytope $\TSP (n)$ (Theorem \ref{theorem_TSP_clique}).
	We recall that the clique number of the skeleton serves as the lower bound for the computational complexity in the class of direct-type algorithms \cite{Bondarenko-Maksimenko}.
	It should also be noted that the value $\Theta (n^{2})$ of clique number correlates with the time complexity $O(n^2)$ of dynamic programming for the pyramidal traveling salesperson problem in a complete graph $K_{n}$ \cite {Klyaus, Gilmore}.
	
	\section{Conclusion}
	\indent
	
	The results presented in the paper, along with those obtained earlier for other combinatorial problems, indicate the existence of a connection between the characteristics of a skeleton of the polytope and the complexity of the corresponding problem.
	So, for polynomially solvable problems like the minimum cut, spanning tree, shortest path and a number of others, the skeletons are completely described and have polynomial clique numbers \cite{Belov, Bondarenko-Maksimenko, Bondarenko-Nikolaev-2016}.
	While for NP-hard problems like the maximum cut, spanning tree with the constraints on the number of leaves and the degree of vertices, the longest path and many others, exponential lower bounds on the clique numbers of the skeletons of the associated polytopes are established \cite{Bondarenko-Maksimenko, Bondarenko-Nikolaev-2016, Bondarenko-NS}.
	And for some problems, such as the traveling salesperson and the knapsack, even the vertex adjacency test is an NP-complete problem \cite {Geist, Papadimitriou}.
	
	Thus, the pyramidal tours polytope considered in this paper is much closer in its polyhedral properties to polytopes of other polynomially solvable problems, such as the spanning tree and the shortest path, and differs from the polytope of the general traveling salesperson problem.

	

\end{document}